\documentclass[11pt]{article}
\usepackage{latexsym}
\usepackage{epsfig,enumerate,bbm}
\usepackage{amsmath,amsthm,amssymb}
\usepackage[usenames]{color}\usepackage{graphicx}
\numberwithin{equation}{section}
\definecolor{brown}{cmyk}{0, 0.72, 1, 0.45}
\definecolor{grey}{gray}{0.5}

\renewcommand{\epsilon}{\varepsilon}

\newcounter{rot}

\newtheorem*{conjecture*}{Conjecture}
\newtheorem{theorem}{Theorem}[section]
\newtheorem*{theorem*}{Theorem}

\newtheorem{claim}[theorem]{Claim}

\newtheorem{corollary}[theorem]{Corollary}

\newtheorem{definition}[theorem]{Definition}

\newcommand{\scr}{\mathcal}

\newcommand{\abs}[1]{\left| #1 \right|}

\newcommand{\sqbs}[1]{\left[ #1 \right]}
\newcommand{\braces}[1]{\left\{ #1 \right\}}
\newcommand{\ignore}[1]{}

\newcommand{\beq}[1]{\begin{equation}\label{#1}}
\newcommand{\eeq}{\end{equation}}
\newcommand{\req}[1]{(\ref{#1})}

\newcommand{\vect}[1]{\vec{\boldsymbol{#1}}}

\def\cH{{\cal H}}
\def\sE{\scr{E}}
\def\sH{\scr{H}}

\title{On the Maximum Number of Edges in a Hypergraph with a Unique Perfect Matching }

\date{}
\author{
{\Large{Deepak Bal\thanks{Department of Mathematical Sciences, Carnegie Mellon University, Pittsburgh, PA 15213, USA} \quad
Andrzej Dudek\thanks{Department of Mathematics, Western Michigan University, Kalamazoo, MI 49008, USA} \quad
Zelealem B. Yilma\footnotemark[1]}}
}

\begin{document}
\maketitle

\begin{abstract}
In this note, we determine the maximum number of edges of a $k$-uniform hypergraph, $k\ge 3$, with a unique perfect matching. This settles a conjecture proposed by Snevily.
\end{abstract}

\section{Introduction}
Let $\sH = (V,\sE)$, $\sE\subseteq \binom{V}{k}$, be a $k$-uniform hypergraph (or $k$-graph) on $km$ vertices for $m\in\mathbb{N}$.
A perfect matching in $\sH$ is a collection of edges 
$\braces{M_1,M_2,\ldots,M_m} \subseteq \sE$ 
such that $M_i \cap M_j = \emptyset$ for all $i\neq j$  and $\bigcup_i M_i = V$. 
In this note we are interested in the maximum number of edges of a hypergraph $\cH$ with a unique perfect matching. Hetyei observed (see, \textit{e.g.}, \cite{HV,L,LP}) that for ordinary graphs (\textit{i.e.} $k=2$), this number cannot exceed $m^2$. To see this, note that at most two edges may join any pair of edges from the matching. Thus the number of edges is bounded from above by $m+2\binom{m}{2} = m^2$.
Hetyei also provides a unique graph satisfying the above conditions. His construction can be easily generalized to uniform hypergraphs (see Section \ref{constr} for details). Snevily~\cite{HS} anticipated that such generalization is optimal. 
Here we present our main result.

\begin{theorem}\label{mainthm}
For integers $k\geq 2$ and $m\geq 1$ let
$$
f(k,m) = m + b_{k,2}\binom{m}{2} + b_{k,3}\binom{m}{3} + \cdots + b_{k,k} \binom{m}{k},
$$
where 
$$
 b_{k,\ell} = \frac{\ell-1}{\ell}\sum_{i=0}^{\ell-1}(-1)^{i}\binom{\ell}{i}\binom{k(\ell-i)}{k}.
$$
Let $\sH=(V,\sE)$ be a $k$-graph of order $km$ with a unique perfect matching. Then 
\begin{equation}\label{eq:ineq}
|\sE| \le f(k,m).
\end{equation}
Moreover, \eqref{eq:ineq} is tight.
\end{theorem}

\noindent
In particular, if $\sH=(V,\sE)$ is a 3-uniform hypergraph of order $3m$ with a unique perfect matching, then 
$$
|\scr{E}| \le f(3,m) = m + 9\binom{m}{2} + 18\binom{m}{3} = \frac{5 m}{2} - \frac{9 m^2}{2} + 3 m^3. 
$$

\section{Construction}\label{constr}
In this section, we provide a recursive construction of a hypergraph $\sH^*_m$ of order $km$ with a unique perfect matching and containing exactly $f(k,m)$ edges.

Let $\sH^*_1$ be a $k$-graph on $k$ vertices with exactly one edge. Trivially, this graph has a unique perfect matching. Suppose we already constructed a $k$-graph $\sH^*_{m-1}$ on $k(m-1)$ vertices with a unique perfect matching. To construct the graph $\sH^*_m$ on $km$ vertices, add $k-1$ new vertices to $\sH^*_{m-1}$ and 
add all edges containing at least one of these new vertices. 
Then, add another new vertex and draw the edge containing the $k$ new vertices. 
Formally, let 
\begin{equation}\label{def:M}
M_i = \braces{k(i-1)+1,\ldots,ki} \text{ for } i=1,\ldots,m.
\end{equation}
Let $\sH^*_m=(V_m,\sE_m)$ , $m\ge 1$, be a $k$-graph on $km$ vertices with the vertex set
$$
V_m = \braces{1,\ldots,km} = \bigcup_{i=1}^{m}M_i
$$
and the edge set (defined recursively)
$$
\sE_m = \sE_{m-1} \cup \braces{E\in \binom{V_m}{k} : E\cap M_m \neq \emptyset,\, km\not\in E} \cup \braces{M_m},
$$
where $\sE_0=\emptyset$.

Note that $\sH^*_m$ has a unique perfect matching, 
namely, $\scr{M}_m = \braces{M_1,M_2,\ldots,M_m}$. 
To see this, observe that the vertex $km$ is only included in 
edge $M_m$. Hence, any matching must include $M_m$. 
Removing all vertices in $M_m$, we see that $M_{m-1}$ must be also included and so on. 
We call the elements of $\scr{M}_m$, {\em matching edges}. 
\begin{claim}
\label{constructionclaim} The $k$-graph $\sH^*_m = (V_m,\sE_m)$ satisfies 
$\abs{\sE_m} = f(k,m)$.
\end{claim}
\begin{proof}
For $\ell=1,2,\ldots,k$, let $\scr{B}_\ell$ be the set of edges that intersect exactly $\ell$ matching edges, \textit{i.e.},
\[ 
\scr{B}_\ell = \braces{E\in\sE_m  : \sum_{i=1}^m{\bf{1}}_{E\cap M_i \neq \emptyset}=\ell}. 
\]
Note that $\sE_m = \bigcup_\ell \scr{B}_\ell$. Clearly, $\abs{\scr{B}_1} = \abs{\{M_1, \ldots, M_m\}}=m$, giving us the first term in $f(k,m)$. 
Now we show that $\abs{\scr{B}_\ell} = b_{k,\ell} \binom{m}{\ell}$
for $\ell = 2,\dots,k$. 
Let $\scr{L} = \braces{M_{i_1},M_{i_2}\ldots,M_{i_\ell}} \subseteq \scr{M}_m$ 
be any set of $\ell$ matching edges
with $1\leq i_1 < i_2 <\cdots< i_\ell \leq m$.
Let $\scr{G}$ be the collection of $k$-sets on the vertex set of $\scr{L}$ which intersect all of $M_{i_1},\ldots,M_{i_\ell}$.
The principle of inclusion and exclusion (conditioning on the number of $k$-sets that do not intersect a given subset of matching edges) yields that
\[
 \abs{\scr{G}} = \sum_{i=0}^{\ell-1}(-1)^{i}\binom{\ell}{i}\binom{k(\ell-i)}{k}.
\]

Now note that due to the symmetry of the roles of the vertices in $\scr{G}$, each vertex belongs to the same number of edges of $\scr{G}$, say $\eta$. Consequently, the number of pairs $(x,E)$, $x\in E \in \scr{G}$ equals $k\ell\eta$. On the other hand, since every edge of $\scr{G}$ consists of $k$ vertices we get that the number of pairs is equal to $\abs{\scr{G}}k$, implying that $\eta = \abs{\scr{G}}/\ell$. 

By construction, 
$E \in \scr{G}$ implies
$E \in \scr{B}_\ell$
unless vertex $ki_\ell$ is in $E$. As
\[\abs{\{ E \in \scr{G} : ki_\ell \in E\}} = \eta = \abs{\scr{G}}/\ell,\]
the number of edges of $\scr{B}_\ell$ on the vertex set of $\scr{L}$ equals
\begin{equation}
\label{blkG} 
\frac{\ell-1}{\ell}\abs{\scr{G}}
= b_{k,\ell}.
\end{equation}
As this argument applies to any choice of $\ell$ matching edges, 
we have $\abs{\scr{B}_\ell} = b_{k,\ell}\binom{m}{\ell}$,
thus proving the claim.
\end{proof}

\begin{corollary}
For all integers $k\geq 2$ and $m\geq 1$,
 \[f(k,m) = m + \sum_{i=1}^{m-1}\sqbs{\binom{k(i+1) - 1}{k} - \binom{ki}{k}}.\]
\end{corollary}
\begin{proof}
We prove this by counting the edges of $\sH^*_m=(V_m,\sE_m)$ in a different way. Let $a_{m} = \abs{\sE_m}$, $m\ge 1$. Then it is easy to see that the following recurrence relation holds: 
$a_1 = 1$ and 
\begin{align}\label{eq:rec}
a_{m} &= a_{m-1} + \binom{km-1}{k} - \binom{k(m-1)}{k} + 1 \text{ for } m\geq2,
\end{align}
where the first binomial coefficient counts all the edges that do not contain vertex $km$; the second coefficient counts all the edges which do not intersect the matching edge $M_m$ (\textit{cf.}~\eqref{def:M}); and the term 1 stands for $M_m$ itself.
Summing \eqref{eq:rec} over $m,m-1,\dots,2$ gives the desired formula.
\end{proof}

Note that $\scr{H}_m^*$ proves that \req{eq:ineq} is tight. 
However, in contrast to the case of $k=2$, 
there are hypergraphs on $km$ vertices 
 containing a unique perfect matching and $f(k,m)$ edges
which are
not isomorphic to $\scr{H}_m^*$.
For example, if $m=2$,
consider an edge $E \in \scr{H}_2^*$, $E \neq M_1, M_2$.
Let $\bar{E}$ be the complement of $E$,
\textit{i.e.}, $\bar{E} = \{1,\ldots,2k\} \setminus E$.
Then, the hypergraph obtained from $\scr{H}_2^*$ by replacing $E$
with $\bar{E}$ provides a non-isomorphic example for the tightness of \req{eq:ineq}.

\section{Proof of Theorem \ref{mainthm}}
We start with some definitions. We use the terms ``edge'' and ``$k$-set'' interchangeably.

\begin{definition}\label{coveringdef}
 Given any collection of $2\leq \ell \leq k$ disjoint edges 
 $\scr{L} = \braces{M_1,\ldots,M_\ell}$, 
 we call a collection of edges $\scr{C} = \braces{C_1,\ldots,C_\ell}$ 
 a {\em covering of $\scr{L}$} if
\begin{itemize}
 \item $C_i\cap M_j \neq \emptyset$ for all $i,j\in\braces{1,\ldots,\ell}$, and 
 \item $\bigcup_iC_i = \bigcup_iM_i$.
\end{itemize}
\end{definition}
\noindent
Note that the second condition forces the edges in a covering to be disjoint. 

\begin{definition}\label{typedef}\
Let $\scr{L}$ be as in Definition \ref{coveringdef}, 
let $\scr{C}$ be a covering of $\scr{L}$ and let $C\in\scr{C}$. 
We say $C$ {\em is of type $\vect{a}$} if 
\begin{itemize}
 \item $\vect{a} = (a_1,\ldots,a_\ell) \in \mathbb{N}^\ell$, 
 $\sum_i a_i = k$ and $a_1\geq a_2\geq\ldots\geq a_\ell \geq 1$, and
 \item there exists a permutation $\sigma$ of $\{1,2,\dots,\ell\}$ 
 such that $\abs{C\cap M_{\sigma(i)}} = a_i$ for each $1\le i\le \ell$.
\end{itemize}

\noindent
Let  $\scr{A}_{k,\ell} = \{ \vect{a} = (a_1,\ldots,a_\ell) \in \mathbb{N}^\ell : 
a_1\geq a_2 \geq \ldots \geq a_\ell \geq 1
\text{ and } a_1 + \ldots + a_\ell = k \}$. 
\end{definition}

Given a vector $\vect{a} \in \scr{A}_{k,\ell}$, 
let $\scr{C}_{\vect{a}}$ be the collection of 
all coverings $\scr{C}$ of $\scr{L}$ such that every 
$C\in \scr{C}$ is of type $\vect{a}$. 
In other words, $\scr{C}_{\vect{a}}$ consists of coverings 
using only edges of type $\vect{a}$. 
We claim that $\scr{C}_{\vect{a}}$ is not empty
for every $\vect{a} \in \scr{A}_{k,\ell}$. 
Indeed, for $i=0,\ldots,\ell-1$ let $\sigma_i$ be a permutation of $\{1,2,\dots,\ell\}$ (clockwise rotation)
obtained by a cyclic shift by $i$, \textit{i.e.},
$\sigma_i(j) = j+i\ (\text{mod } \ell)$.  
We form $C_i$ by picking $a_{\sigma_i(j)}$
items from $M_j$ for each $1\le j\le \ell$.
As $\sum_i a_{\sigma_i(j)} = k$,
we may pick the $\ell$ edges $C_i$ to be disjoint,
thereby obtaining a covering.

\begin{proof}[Proof of Theorem \ref{mainthm}]  
Let $\sH=(V,\sE)$ be a $k$-graph of order $km$ with the unique perfect matching
$\scr{M} = \braces{M_1,\ldots,M_m}$. We show that $|\sE| \le f(k,m)$.

We partition the  edges 
into collections of edges which intersect exactly $\ell$ of the matching edges. 
That is, 
for $\ell = 1,\ldots, k$, we set
\[
 \scr{B}_\ell = \braces{E\in\sE  : \sum_{i=1}^m{\bf{1}}_{E\cap M_i \neq \emptyset}=\ell}.
\]
Clearly, $\abs{\scr{E}} = \sum_{\ell=1}^k \abs{\scr{B}_\ell}.$  Once again, $\abs{\scr{B}_1} = m$.
We will show, by contradiction, 
that $\abs{\scr{B}_\ell} \leq b_{k,\ell} \binom{m}{\ell}$
for all $2\le \ell \le k$.

Suppose that $\abs{\scr{B}_\ell} > b_{k,\ell} \binom{m}{\ell}$
for some $2\le \ell \le k$.
Then, by the pigeonhole principle,
there exists some set of $\ell$ matching edges, say, without loss of generality, 
$\scr{L} = \{M_1,\ldots,M_\ell\}$
such that 
\begin{equation}\label{eq:B}
\abs{\scr{B}_\ell \cap \sH[\scr{L}]} \geq b_{k,\ell}+1,
\end{equation}
where $\sH[\scr{L}]$ denotes the sub-hypergraph of $\sH$ spanned by the vertices in
$\bigcup_{i=1}^{\ell} M_i$.
Let $\scr{G}$ be the collection of all $k$-sets on $\bigcup_i M_i$
that intersect every $M_i \in \scr{L}$. That is
\[\scr{G} = \left\{A : \abs{A}=k, A \cap M_i \neq \emptyset \text{ for each } 1\le i\le \ell 
\text{ and } A \subseteq \bigcup_{i} M_i \right\}.\]
As in \req{blkG}, we have
\[ b_{k,\ell} = \frac{\ell-1}{\ell} \abs{\scr{G}} 
= \frac{\ell-1}{\ell}\sum_{\vect{a} \in \scr{A}_{k,\ell}} \abs{\scr{G}_{\vect{a}} },\]
where $\scr{G}_{\vect{a}}$ is the collection of $k$-sets of type $\vect{a}$.
Hence, by \eqref{eq:B} we get
$$
\abs{\scr{B}_\ell \cap \sH[\scr{L}]} \geq \frac{\ell-1}{\ell}\sum_{\vect{a} \in \scr{A}_{k,\ell}} \abs{\scr{G}_{\vect{a}} } +1,
$$
and consequently,
there exists some type $\vect{a}$ such that
\begin{equation}\label{eq:contr}
\abs{\scr{B}_\ell \cap \scr{G}_{\vect{a}}} 
\geq \frac{\ell-1}{\ell} \abs{\scr{G}_{\vect{a}}} + 1.
\end{equation}
Recall that $|\scr{C}|=\ell$ and that $\scr{C}_{\vect{a}}$ is the nonempty collection of 
all coverings $\scr{C}$ of $\scr{L}$ such that every 
$C\in \scr{C}$ is of type $\vect{a}$. 
By symmetry, every $k$-set $A\in \scr{G}_{\vect{a}}$ belongs to exactly
$$
\frac{|\scr{C}_{\vect{a}}| \ell}{|\scr{G}_{\vect{a}}|}
$$
coverings $\scr{C}\in \scr{C}_{\vect{a}}$. Since no $\scr{C} \in \scr{C}_{\vect{a}}$ is contained in $\scr{H}$ (otherwise we could replace $\scr{L}$
 by $\scr{C}$ to obtain a different perfect matching,  contradicting the uniqueness of $\scr{M}$), the number of $k$-sets in $\scr{G}_{\vect{a}}$ that are not in $\scr{B}_\ell$ is at least
 $$
\abs{\scr{C}_{\vect{a}}}  \left/  \frac{|\scr{C}_{\vect{a}}| \ell}{|\scr{G}_{\vect{a}}|}\right. = \frac{|\scr{G}_{\vect{a}}|}{\ell}.
 $$
That means,
$$
\abs{\scr{B}_\ell \cap \scr{G}_{\vect{a}}} 
\le \frac{\ell-1}{\ell} \abs{\scr{G}_{\vect{a}}}
$$
which contradicts \eqref{eq:contr}. Thus, $\abs{\scr{B}_\ell} \leq b_{k,\ell} \binom{m}{\ell}$, as required.
\end{proof}


\begin{thebibliography}{99}

\bibitem{HV}
A. Hoffmann, L. Volkmann.
\newblock On unique $k$-factors and unique $[1,k]$-factors in graphs,
\newblock {\em Disc. Math.} {\bf 278} (2004), 127--138.

\bibitem{L}
{L. Lov\'asz},
\newblock {On the structure of factorizable graphs. I, II},
\newblock {\em Acta Math. Acad. Sci. Hungar.} {\bf 23} (1972), 179--195; ibid. {\bf 23} (1972), 465--478.

\bibitem{LP}
{L. Lov\'asz, M. Plummer},
\newblock{{\em Matching Theory.}}
\newblock{ Corrected reprint of the 1986 original.  AMS Chelsea Publishing, Providence, RI, 2009.} 

\bibitem{HS} H.~Snevily, personal communication, 2010.

\end{thebibliography}
\end{document}